\documentclass[a4paper,11pt,reqno]{amsart}

\usepackage[T2A]{fontenc}
\usepackage[cp1251]{inputenc}
\usepackage[english]{babel}
\usepackage{amsmath,amsthm,amsfonts,amssymb}
\usepackage{enumerate}

\theoremstyle{plain}
\newtheorem{theorem}{Theorem}
\newtheorem{lemma}{Lemma}
\newtheorem{corollary}{Corollary}
\newtheorem*{corollary*}{Наслідок}

\theoremstyle{definition}

\theoremstyle{remark}

\newtheorem*{remark*}{Зауваження}

\newcommand{\N}{\mathbb{N}}

\newcommand{\abs}[1]{\left\lvert#1\right\rvert}

\newcommand{\set}[1]{\left\{#1\right\}}

\newcommand{\bO}[1]{\bar{\mathrm{O}}^{#1}}
\newcommand{\bcyl}[2]{\bar{\mathrm{O}}^{#1}_{[#2]}}

\makeatletter
\newcommand{\Cset}[2][\bO1]{C[#1,
  \if\relax\expandafter\@gobble#2\relax #2\else\{#2\}\fi]}
\newcommand{\CsetI}[2][\bO1]{C^*[#1,
  \if\relax\expandafter\@gobble#2\relax #2\else\{#2\}\fi]}
\makeatother

\uccode`ґ=`Ґ \uccode`Ґ=`Ґ \lccode`Ґ=`ґ \lccode`ґ=`ґ
\uccode`є=`Є \uccode`Є=`Є \lccode`Є=`є \lccode`є=`є
\uccode`і=`І \uccode`І=`І \lccode`І=`і \lccode`і=`і
\uccode`ї=`Ї \uccode`Ї=`Ї \lccode`Ї=`ї \lccode`ї=`ї

\makeatletter

\def\author@andify{%
  \nxandlist {\unskip ,\penalty-1 \space\ignorespaces}%
    {\unskip ,\penalty-2 \space}%
    {\unskip ,\penalty-2 \space}}
\makeatother

\begin{document}

\title[$\bO1$ expansion: dynamical, probabilistic and fractal points of view]
{Ostrogradsky-Sierpi\'nski-Pierce expansion: dynamical systems, probability theory and fractal geometry points of view.}

\author[S.Albeverio, G.Torbin]{S.Albeverio$^{1,2,3,4,5}$, G.Torbin$^{6,7}$}


\begin{abstract}
We establish several new probabilistic, dynamical, dimensional and number theoretical phenomena connected with  Ostrogradsky-Sierpi\'nski-Pierce expansion.

First of all, we develop metric, ergodic and dimensional theories of the Ostrogradsky-Sierpi\'nski-Pierce expansion.  In particular, it is proven that for Lebesgue almost all real numbers any digit $i$ from the  alphabet $A= \mathbb{N} $ appears  only finitely many times  in the difference-version of the Ostrogradsky-Sierpi\'nski-Pierce expansion.

Properties of the symbolic dynamical system generated by a shift-transformation $T$ on the difference-version of the Ostrogradsky-Sierpi\'nski-Pierce expansion are also studied in details. It is shown that there are no probability measures which are invariant and ergodic (w.r.t. $T$) and  absolutely continuous (w.r.t. Lebesgue measure).

Thirdly, we study properties of random variables $\eta$ with independent identically distributed differences of the Ostrogradsky-Sierpi\'nski-Pierce expansion. Necessary and sufficient conditions for $\eta$ to be discrete resp. singularly continuous are found. We prove that $\eta$ can not be absolutely continuously distributed.

\end{abstract}

\maketitle

$^1$~Institut f\"ur Angewandte Mathematik, Universit\"at Bonn,
Endenicher Allee 60, D-53115 Bonn (Germany); $^2$~SFB 611, Bonn; $^3$ BiBoS,
Bielefeld--Bonn; $^4$~CERFIM, Locarno and Acc. Arch., USI
(Switzerland); $^5$~IZKS, Bonn; E-mail:
albeverio@uni-bonn.de

$^6$~National Pedagogical University, Pyrogova str. 9, 01030 Kyiv
(Ukraine) $^{7}$~Institute for Mathematics of NASU,
Tereshchenkivs'ka str. 3, 01601 Kyiv (Ukraine); E-mail:
torbin@iam.uni-bonn.de (corresponding author)

\medskip

\textbf{AMS Subject Classification (2010):}  .

\medskip


\section{Introduction}

It is now well known that any real number  $x\in(0,1)$ can be represented in the form
\begin{align}
&\sum_k \frac{(-1)^{k+1}}{q_1q_2\dots q_k}, \quad \text{where $q_k \in \mathbb{N},
q_{k+1}>q_k, ~~ k \in \mathbb{N}$}.\label{eq:o1series}
\end{align}
 If  $x$ is irrational, then this expansion is unique. In the opposite case there are two different expansion of $x$ into series of the above form (\cite{Rem51}).

In the western mathematical literature series of the above form are known as the Pierce series, and in the eastern  literature they are known as the Ostrogradsky series of the first type.  One can find notes on the history of the discovery and the development of such series
in the paper~\cite{PVB99}. Here we would like just to mention that such series can also  be associated with names of Lambert (\cite{Lambert}), Lagrange (\cite{Lagrange}), Sierpi\'nski (\cite{Sie74}).

M.V. Ostrogradsky  was probably the first (1860) who developed a few of numerical properties of such an expansion (see, e.g., \cite{Rem51}).
Some algorithms for the representation of real numbers in positive and alternating series were proposed by W. Sierpi\'nski in \cite{Sie74}. One of these algorithms leads to the series (\ref{eq:o1series}). T.A. Pierce  was probably the first (1929) who used this expansion for a numerical estimation of algebraic roots of polynomials (\cite{Pie29}). In what follows we shall use the notion "Ostrogradsky-Sierpi\'nski-Pierce expansion"  for the above series.

The Ostrogradsky-Sierpi\'nski-Pierce series converges rather quickly, giving a good approximation of irrational numbers by rationals, which are partial sums of the above series.

Let us recall (\cite{ABPT}) that the expression~\eqref{eq:o1series} can be rewritten in the form
\begin{equation}\label{eq:bo1series}
\frac1{g_1}-\frac1{g_1(g_1+g_2)}+\dots+\frac{(-1)^{n-1}}{g_1(g_1+g_2)\dots(g_1+g_2+\dots+g_n)}+\dotsb,
\end{equation}
where
\[
g_1=q_1\quad \text{і}\quad g_{n+1}=q_{n+1}-q_n\quad \text{for all $n\in \mathbb{N}$}.
\]
The expression ~\eqref{eq:bo1series} will be denoted by
\[
\bO1(g_1,g_2,\dots,g_n,\dots).
\]
 and is said to be $\bO1$-expansion (or the Ostrogradsky-Sierpi\'nski-Pierce expansion with independent symbols), and coefficients $g_n=g_n(x)$~are called
$\bO1$-symbols (coefficients) of a real number $x\in(0,1)$. There are several papers on metric theory of this expansion (see, e.g., \cite{ABPT,BPT2006,Sha86, VZ75} and references therein), but they should be considered only as first steps in the development of the general theory like for the continued fractions. There are a lot of common features between these two expansions, but the Ostrogradsky-Sierpi\'nski-Pierce expansion generates essentially more complicated "geometry of cylindrical intervals".
 It is known that the development of  metric and ergodic theories of some expansion for reals can be essentially simplified if one can find a measure which is invariant and ergodic w.r.t. one-sided shift transformation on the corresponding expansion and absolutely continuous w.r.t. Lebesgue measure (see, e.g., \cite{Schweiger_book}).  For instance, having the Gauss measure (i.e., the probability measure with density $f(x)= \frac{1}{\ln 2} \frac{1}{1+x}$ on the unit interval) as invariant and ergodic measure w.r.t. the transformation $T(x)=\frac{1}{x} (mod 1)$, one can easily derive main metric and ergodic properties of continued fraction expansions (see, e.g., \cite{Bil,Khi63,Schweiger_book}).

The main aims of the paper are:

1) to develop  ergodic, metric and dimensional theories of $\bO1$ - expansion for real numbers (in particular, to find normal properties of real numbers, depending on asymptotic frequencies  $\nu_{i}(x, \bO1)$  of $\bO1$-symbols ($i \in \mathbb{N}$), where $\nu_{i}(x, \bO1)= \lim\limits_{n\to \infty} \frac{N_i(x,n)}{n}$,  and  $N_i(x,n)$ is the number of terms "$i$"  among the first $n$ $\bO1$-coefficients of $x$);

2) to study properties of the symbolic dynamical system generated by the one-sided shift transformation on the $\bO1$-expansion:
$$
 \forall ~ x = \bO1(g_1(x), g_2(x), \dots, g_n(x), \dots)~~ \in [0,1],
 $$
  $$T(x) = T(\bO1(g_1(x), g_2(x),\dots, g_n(x),\dots))= \bO1(g_2(x), g_3(x),\dots, g_n(x),\dots);
$$
3) to study  distributions of random variables
 $$\eta =\sum_{k=1}^\infty
\frac{(-1)^{k-1}}{\eta_1(\eta_1+\eta_2)\dots(\eta_1+\eta_2+\dots+\eta_k)}
=\bO1(\eta_1,\eta_2,\dots,\eta_k,\dots),
$$
whose $\bO1$-symbols $\eta_k$ are  independent random variables taking the values $1$, $2$,~$\dots$, $m$,~$\dots$ with probabilities $p_{1k}$, $p_{2k}$,~$\dots$, $p_{mk}$,~$\dots$
respectively, $ p_{mk}\geq0, \quad
\sum\limits_{m=1}^\infty {p_{mk}}=1, ~ \forall k \in \mathbb{N}.
$

\section{Sets $\Cset{V_n}$ and their metric and fractal properties.}

Let $\{V_n\}$~ be a given sequence of non-empty subsets of positive integers. Let us consider the set  $\Cset{V_n}$, which is the closure of  the set $\CsetI{V_n}$ of all irrational numbers
$x=\bO1(g_1(x),g_2(x),\dots,g_n(x),\dots)$ such that  $g_n(x)\in V_n$ for all $n\in \mathbb{N}$.

It is clear that  $\Cset{V_n}$ is a nowhere dense set if and only if the condition $V_n\not= \mathbb{N}$ holds  for an infinite number of  $n$'s.
The set of real numbers whose continued fraction expansion does not contain a given symbol $i \in \mathbb{N}$ is a Cantor-like set of zero Lebesgue measure. Indeed, almost all (in the sense of Lebesgue measure) real numbers contain a given digit $i$ with non-zero asymptotic frequency $\nu^{c.f.}_i =\frac{1}{\ln2} \ln \frac{(i+1)^2}{i(i+2)}$ (see, e.g., \cite{Bil}). So, all points without using the symbol $i$ belong to the exceptional zero-set.
For the Ostrogradsky-Sierpi\'nski-Pierce expansion the metric properties of $\Cset{V_n}$  depend essentially on the sequence of sets $V_k$ of admissible digits.

In \cite{ABPT,BPT2006} some sufficient conditions for the set  $\Cset{V_n}$  to be of zero resp. positive Lebesgue measure are found. We collect here some of these results without proof to be used later in the paper and to stress essential differences in metric theories of continued fractions and the Ostrogradsky-Sierpi\'nski-Pierce expansions.
\begin{theorem} \label{1,2,...n} 
 Let  $V_k=\set{1,2,\dots,m_k}$, $m_k\in\mathbb{N}$.

1) If $
\sum\limits_{k=1}^\infty\frac{m_1+m_2+\dots+m_k}{m_{k+1}}<\infty,
$
then the Lebesgue measure $\lambda(\Cset{V_n})>0$.

2) If $
\sum\limits_{k=1}^\infty\frac{k}{m_{k}} = \infty,
$
then $\lambda(\Cset{V_n})=0$.
\end{theorem}

\textbf{Example.}

1)   If $m_k = 2^{k!}$, then $\lambda(\Cset{V_n})>0$.

2)  If $m_k = k^2$, then $\lambda(\Cset{V_n})=0$.
\bigskip

\begin{theorem} \label{v+1, v+2, ...} 

 Let $V_k=\set{v_k+1, v_k+ 2,\dots }$, $v_k \in \mathbb{N}$.

If
$\sum\limits_{k=1}^\infty \frac{v_k}{2^k}<+\infty $, then $\lambda(\Cset{V_n})>0.
$
\end{theorem}

\begin{corollary}
If $V_k=  V = \{ v+1, v+2, ...\}$,  then  $\lambda(\Cset{V_n}) >0.$
\end{corollary}

In the case of zero Lebesgue measure, the next level of studying of properties of sets $\Cset{V_n}$ is the determination
of their Hausdorff dimension  $\dim_H(\cdot)$ (see, e.g., \cite{Fa} for the definition and main properties of this main fractal dimension).


We shall study this problem for the case where $V_n = \{1,2,..., k_n\}$. A similar problem for the continued fraction expansion were studied by many authors during last  60 years.
Set
 $$E_2 = \{x: x = \Delta^{c.f.}_{\alpha_1(x)...\alpha_k(x)...~} ,\alpha_k(x) \in \{1,2\} \}.$$

In  1941   Good \cite{Good} shows that  $$0,5194 < \dim_H (E_2)
<0, 5433.$$
In 1982 and  1985 Bumby \cite{Bumby82, Bumby85} improves these bounds:
$$0,5312 < \dim_H (E_2) < 0,5314.$$
 In  1989 Hensley \cite{Hensley89}  shows that
 $$0,53128049 < \dim_H (E_2) < 0,53128051.$$
In 1996 the same author  (\cite{Hensley96}) improves his estimate up to $$0,5312805062772051416.$$
New approach to the determination of the Hausdorff dimension of the set $E_2$  with a desired precision was developed by Jenkinson and Policott in 2001 \cite{Jenkinson}.

Our nearest aim is to study fractal properties of sets which are $\bO1$-analogues of the above discussed set $E_2$, i.e., the set
$$
\bO1_2 = \{x: x= \bO1(g_1(x) g_2(x)...g_k(x)...), g_{k}(x) \in
\{1,2\}\}
$$ and their generalization $$
\bO1_n = \{x: x= \bO1(g_1(x) g_2(x)...g_k(x)...), g_{k}(x) \in
\{1,2,..., n\}\}.
$$
Firstly, let us mention, that from Theorem \ref{1,2,...n} it follows that all these sets are of zero Lebesgue measure, which is similar to the c.f.-case.
But the following theorem shows that from the fractal geometry point of view  the sets $E_2$ and $\bO1_2$ (as well as their generalizations) are cardinally different.

\begin{theorem}\label{theorem pro rozmirnist HB dlya ryadiv Ostrogr}
For any $n \in \mathbb{N}$ the Hausdorff dimension of the set
$\bO1_n$ is equal to zero.
\end{theorem}
\begin{proof}
Let $\bcyl1{c_1c_2\dots c_k}$ be the cylindrical interval of the $\bO1$-expansion, i.e., the closure of all real numbers $x$ from the unit interval such that $g_i(x)=c_i, i=1,2,...,k$. It is known (\cite{ABPT}) that  $\abs{\bcyl1{c_1c_2\dots c_k}} = \frac{1}{\sigma_1
\sigma_2 ... \sigma_k (\sigma_k + 1)}$,  where $\sigma_j = c_1 + c_2 +
... + c_j$. Therefore,  $\abs{\bcyl1{c_1c_2\dots c_k}} \leq \abs{\bcyl1{1
1 \dots 1}} = \frac{1}{k!(k+1)}$.

Let us fix a positive real number $\alpha$. It is clear that the set $\bO1_n$  is contained in the union of the following cylinders:
 $$
 \bO1_n \subset \bigcup_{i_1 = 1}^{n} \bigcup_{i_2 = 1}^{n} ... \bigcup_{i_k = 1}^{n} \bcyl1{c_1 c_2\dots c_k}, ~~~ \forall k \in \N,
  $$
which forms it's  $\varepsilon_k = \frac{1}{k!(k+1)}$-covering.
The $\alpha$-volume of this covering is equal to  $n^k \cdot \left(\frac{1}{k!
(k+1)}\right)^{\alpha}$.  So, the Hausdorff pre-measure
$$ H^{\alpha}_{\varepsilon_k}(\bO1_n) := \inf\limits_{|E_i|\leq \varepsilon_k} \sum\limits_{i}|E_i|^{\alpha} \leq n^k \cdot \left(\frac{1}{k! (k+1)}\right)^{\alpha} \to 0 (k \to \infty),~~ \forall \alpha>0.$$
Therefore, $H^{\alpha}_{\varepsilon_k} (\bO1_n) =0, ~\forall k \in \mathbb{N}, ~ \forall \alpha>0$.

 So, $H^{\alpha}(\bO1_n) = \lim\limits_{k\to \infty}H^{\alpha}_{\varepsilon_k} (\bO1_n) =0, ~ \forall \alpha>0$ and, hence, $$\dim_H (\bO1_n):= \inf \{ \alpha: H^{\alpha} (\bO1_n) =0 \} =0,$$ which proves the theorem.
\end{proof}

Let  $B(\bO1)$ be the set of all real numbers from the unit interval with bounded  $\bO1$-symbols (i.e., $x\in B(\bO1)$ iff there
exists a positive integer $K_x$ (depending on $x$) such that $g_k(x)\leq K_x$ for all $k\in \N$).

\textbf{Corollary 1.} The set $B(\bO1)$ with bounded  $\bO1$-symbols is an anomalously fractal set, i.e.,
$$\dim_H (B(\bO1)) =0.$$

\textbf{Corollary 2.} For all but of the Hausdorff dimension zero real numbers  $x \in [0,1]$ the sequences  $\{g_k(x)\}$ of their  $\bO1$-symbols are unbounded.

\textbf{Remark.} The set of real numbers with bounded continued fraction symbols is of full Hausdorff dimension ($\dim_H (B(c.f.)) = 1$), which stresses essential differences also  in dimensional theories of the Ostrogradsky-Sierpi\'nski-Pierce and continued fraction expansions.

\section{Properties of the symbolic dynamical system generated by the Ostrogradsky-Sierpi\'nski-Pierce expansion}

 Let us consider a dynamical system which is  generated by one-sided shift transformation  $T$ on the $\bO1$-expansion:
$$
 \forall ~ x = \bO1(g_1(x), g_2(x), \dots, g_n(x), \dots)~~ \in [0,1],
 $$
  $$T(x) = T(\bO1(g_1(x), g_2(x),\dots, g_n(x),\dots))= \bO1(g_2(x), g_3(x),\dots, g_n(x),\dots).
$$

 Recall that a set  $A$ is said to be invariant w.r.t. a measurable transformation  $T,$ if $A=T^{-1}A.$
  A measure $\mu$ is said to be ergodic w.r.t. a transformation  $T,$ if any invariant set  $A\in\mathfrak{B}$ is either of full or of zero measure $\mu$. A measure $\mu$ is said to be invariant w.r.t. a transformation  $T,$ if for any set  $E\in\mathfrak{B}$ one has
  $\mu(T^{-1}E)=\mu(E)$.

Let us remind that to develop metric and ergodic theories of any expansion it would very desirable to have a measure which is T-invariant, T-ergodic and absolutely continuous w.r.t. the Lebesgue measure (i.e., to find an analogue of the Gauss measure for c.f.-expansion). Unfortunately, the following  theorem shows that the above mentioned ergodic approach is not applicable for the Ostrogradsky-Sierpi\'nski-Pierce expansion.

\begin{theorem}
  There are no probability measures which are  simultaneously invariant and ergodic w.r.t. the  one-sided shift transformation  $T$ on $\bO1$-expansion, and  absolutely  continuous w.r.t. the Lebesgue measure.
\end{theorem}
\begin{proof}
Firstly we prove the lemma  characterizing generic properties of asymptotic frequencies of digits (from the alphabet) in the Ostrogradsky-Sierpi\'nski-Pierce expansion of real numbers.
\begin{lemma} \label{theorem  pro rivnis chastot 0 dlya O1}
Let $\nu_i (x, \bO1)$ be the asymptotic frequency of a symbol $i$  in the $\bO1$-expansion of $x$ (if the limit $\lim\limits_{k \to \infty} \frac{N_{i}(x,k)}{k}$ exists). Then for Lebesgue almost all real numbers $x \in [0,1]$ and for any  symbol  $i \in \mathbb{N}$ the asymptotic frequency  $\nu_i (x, \bO1)$ is equal to zero.
\end{lemma}
\begin{proof}
Let $x$ be the random variable which is uniformly distributed on the unit interval, i.e., the Lebesgue measure coincides with the probability measure $\mu_{x}$. Let $i$ be a given positive integer, and let us consider the following sequence of random variables:
$$\xi_{k}= \xi_{k}(x) = 0,~~~ \mbox{if}~~~ g_{k}(x) \neq i;$$
$$\xi_{k}= \xi_{k}(x) = 1,~~~ \mbox{if}~~~ g_{k}(x) = i.$$

It is clear that $N_{i}(x,k) = \xi_{1}(x) + \xi_{2}(x)+ ... + \xi_{k}(x)$. Let $$G_i =\{x: \lim\limits_{k \to \infty} \frac{N_{i}(x,k)}{k} = 0\}. $$
The event  $x \in G_i $ does not depend on any finite number of $\bO1$-symbols of  $x$. Therefore,  either $\mu_{x}(G_i)=0$ or $\mu_{x}(G_i)=1$.

  Fix $V_{n}= V=\{ i+1, i+2,... \}$. If $x \in \Cset{V_n} $,  then $N_{i}(x,k) = 0, \forall k \in \mathbb{N}$.  Therefore,  $\Cset{V_n} \subset G_i$. From the corollary of Theorem \ref{v+1, v+2, ...} it follows directly that  $\lambda(\Cset{V_n}) >0$. So,  $\mu_{x}(G_1) = \lambda(G_1)=1$.
\end{proof}

To prove the theorem ad absurdum, let us assume that there exists an absolutely continuous probability measure $\nu$, which is invariant and ergodic w.r.t. the above defined transformation   $T$.  Then, by Birkhoff ergodic theorem,  for  $\nu$-almost all $x \in [0,1]$
  and for any function $\varphi \in L^1([0,1], d\nu)$ we get:
  $$
 \lim_{n\to \infty} \frac{1}{n} \sum_{j=0}^{n-1} \varphi(T^j(x)) = \int_{0}^{1} \varphi(x) d(\nu(x))= \int_{0}^{1} \varphi(x) f_{\nu}(x) dx,
 $$
 where $f_{\nu}(x)$ is the density of $\nu$.

  Choose $\varphi_{і} (x) = 1$, if $x\in \bO1_{[i]}$, and $\varphi_i (x)= 0$ otherwise. Then $$\int_{0}^{1} \varphi_i(x) f_{\nu}(x) dx = \int_{\bO1_{[i]}} f_{\nu}(x) dx >0 ~~\mbox{for at least one} ~~i \in \mathbb{N}.$$
 Let the latter condition holds for the index $i_{0}$.

 On the other hand, from the above Lemma it follows that   $$\lim_{n \to \infty} \frac{1}{n} \sum_{j=0}^{n-1} \varphi_{i_0}(T^j(x)) = \lim_{n \to \infty}\frac{N_{i_0}(x,n)}{n} = 0$$
  for  $\lambda$-almost all $x \in [0,1]$.

 Hence, $$\lim_{n \to \infty}\frac{N_{i_0}(x,n)}{n} = 0$$
  for  $\lambda$-almost all $x \in [0,1]$,  and simultaneously
  $$\lim_{n \to \infty}\frac{N_{i_0}(x,n)}{n} > 0$$ for the set of positive Lebesgue measure. This contradiction proves the theorem. \end{proof}
  \textbf{Remark.} From the proof given above it follows that there are no probability measures which are  simultaneously invariant and ergodic w.r.t. the  one-sided shift transformation  $T$ on $\bO1$-expansion, and   which contains an absolutely  continuous component in its Lebesgue decomposition.

This result can be naturally applied to study the Lebesgue structure of  the random Ostrogradsky-Sierpi\'nski-Pierce expansion, i.e., the random variable
\begin{equation}\label{eq:eta}
\eta=\sum_{k=1}^\infty
\frac{(-1)^{k-1}}{\eta_1(\eta_1+\eta_2)\dots(\eta_1+\eta_2+\dots+\eta_k)}
=\bO1(\eta_1,\eta_2,\dots,\eta_k,\dots),
\end{equation}
whose $\bO1$-symbols $\eta_k$ are  independent identically distributed random variables taking values  $1$, $2$,~$\dots$, $m$,~$\dots$ with probabilities  $p_{1}$, $p_{2}$,~$\dots$, $p_{m}$,~$\dots$
respectively, i.e.,
\[
P\set{\eta_k=m}=p_{m} \quad \text{with} \quad p_{m}\geq0, \quad
\sum\limits_{m=1}^\infty {p_{m}}=1 \quad \forall k \in \mathbb{N}.
\]


\begin{theorem}
Let  $\{\eta_k\}$  be  a sequence of independent identically distributed random variables taking values  $1$, $2$,~$\dots$, $m$,~$\dots$ with probabilities  $p_{1}$, $p_{2}$,~$\dots$, $p_{m}$,~$\dots$
respectively. Then the random variable $\eta$ defined by (\ref{eq:eta}) has either:

1) degenerate distribution (if $p_i =1$ for some $i \in \mathbb{N}$);

2) or pure singularly continuous distribution  (in all other cases).
\end{theorem}
\begin{proof}
1) The correctness of the first assertion follows directly from the necessary and sufficient condition for discreteness of $\eta$ in general independent case (see, e.g., \cite{ABPT}): random variable  $\eta$ is purely discretely distributed if and only if  $\prod\limits_{k=1}^{\infty} \max\limits_i p_{ik} >0$.

2) Let us prove that in the case of continuity the distribution of $\eta$ does not contain absolutely continuous component.  To this end we need an auxiliary lemma.
 \begin{lemma} If $\{\eta_k\}$ are independent and identically distributed random variables, then the measure $\mu_{\eta}$  is invariant and ergodic w.r.t. the one-sided shift transformation  $T$.
  \end{lemma}
 \begin{proof}
  1) Let  $A$ be an invariant set w.r.t. $T$. Then $T(T^{-1}A)= T(A)$ and, so, $A=TA$.
  Therefore $A=T^{-1}A=T^{-1}(TA)$.

  If $x =\bO1(g_1(x) g_2(x)... g_k(x)...)$ and  $x \in A$, then $$T^{-1}(T x) = \{ x: x= \bO1(c_1 g_2(x)... g_k(x)...), c_1 \in N\} \subset A.$$

Therefore the event $\{x \in A\}$ does not depend on the first  $\bO1$-symbol of the point $x$. Similarly one can show that this event does not
 depend on the initial $n$   $\bO1$-symbols of  $x$. Then, from Kolmogorov's "zero and one" law it follows that either
  $\mu_{\eta}(A)=0$ or $\mu_{\eta}(A)=1$. So, $\mu_{\eta}$ is ergodic w.r.t.  $T$.

2) Since the Borel  $\sigma$-algebra $\mathcal{B}$ is generated by the family of  $\bO1$-cylinders, i.e., sets of the form $\bO1_{[c_1 c_2 ...c_n]}$, it is sufficient to show that the measure $\mu_{\eta}$ is invariant on these cylinders (\cite{Bil}). It is clear that  $\mu_{\eta}(\bO1_{[c_1 c_2 ...c_n]})=
p_{c_1}\cdot p_{c_2} \cdot ... \cdot p_{c_n}$. Since
$T^{-1} (\bO1_{[c_1 c_2 ...c_n]}) = \bO1_{[i c_1 c_2 ...c_n]}, i \in \mathbb{N},$
 we have$$\mu_{\eta}(T^{-1} (\bO1_{[c_1 c_2 ...c_n]})) =  \sum\limits_{i=1}^{\infty} \mu_{\eta}(\bO1_{[ i c_1 c_2 ...c_n]}) =$$ $$ =p_{c_1}\cdot p_{c_2} \cdot ... \cdot p_{c_n} \sum\limits_{i=1}^{\infty} p_i = p_{c_1}\cdot p_{c_2} \cdot ... \cdot p_{c_n} = \mu_{\eta}(\bO1_{[c_1 c_2 ...c_n]}), $$ which proves the lemma.
 \end{proof}

 Let us choose a positive integer  $i_0$ such that $p_{i_0} >0$ and consider the set $M_{i_0} = \{x: x \in [0,1], \nu_i(x, \bO1) = p_{i_0}>0 \}$. Since symbols of $\bO1$-expansion are independent w.r.t. the measure $\mu_{\eta}$, from the strong law of large number it follows that  this set is of full $\mu_{\eta}$-measure.

Let us now consider the set  $L^{*}_{i_0} = \{x: x \in [0,1], ~ \nu_{i_0}(x, \bO1) = 0 \}$.  From Lemma \ref{theorem  pro rivnis chastot 0 dlya O1} it follows directly that $ \lambda(L^{*}_{i_0}) =1$. The sets   $ M_{i_0}$ and $L^{*}_{i_0}$ have no mutual intersection. The first one is a support of the probability measure $\mu_{\eta}$, and the second one is a support of Lebesgue measure on the unit interval. So, $\mu_{\eta} \bot \lambda$, which completes the proof of the theorem.
\end{proof}

\textbf{Corollary.} The random variable  $\eta$ with independent identically distributed increments of the Ostrogradsky-Sierpi\'nski-Pierce expansion has a pure distribution, and it is can not be absolutely continuous.

\section{On normal properties of reals in $\bO1$ - expansion and singularity of  random Ostrogradsky-Sierpi\'nski-Pierce expansions in general independent case}

A property "$\Upsilon$" of real numbers is said to be normal, if it holds for almost all (in the sense of the Lebesgue measure) real numbers.
Typical normal properties are  "to be irrational", "to be transcendental" which do not depend on a chosen system of numeration (expansion). Having a fixed expansion, it is convenient to formulate normal properties via properties of symbols (digits) of this expansion.
For instance, for the classical decimal expansion the following properties are normal: "to have infinitely many zeroes (in the expansion)", " does not contain any period", "to contain any digit from the alphabet with the asymptotic frequency $\frac{1}{10}$".
For the continued fractions expansion as an example of typical normal property one may consider  "to contain a symbol $i$ from the alphabet with the asymptotic frequency $\frac{1}{\ln 2} \ln \frac{(i+1)^2}{i(i+2)}$" (see, e.g., \cite{Bil} for details and other examples).
The investigation of normal properties of real numbers written in some expansion is an important part in the development of the metric theory of the corresponding expansion, because to determinate Lebesgue measure (or any other equivalent measures) of a given subset, one may ignore  real numbers loosing normal properties. They are also extremely helpful for the studying of properties of probability distributions connected to the corresponding expansion.

In the initial sections of our paper we already derived two normal  properties of real numbers written via $\bO1$ - expansion:

1) for Lebesgue almost all  real numbers  $x \in [0,1]$ the sequences  $\{g_k(x)\}$ of their  $\bO1$-symbols are unbounded;

2) for Lebesgue almost all real numbers $x \in [0,1]$ and for any  symbol  $i \in \mathbb{N}$ the asymptotic frequency  $\nu_i (x, \bO1)$ is equal to zero.

The following theorem gives us rather unusual property of the $\bO1$-expansion and it can be considered as an essential strengthening of the latter property.
\begin{theorem} \label{theorem  pro skinchennist povtoriv synvoliv v rozkladi Ostr-Pierce}
For Lebesgue almost all real numbers $x \in [0,1]$ and for any  symbol  $i \in \mathbb{N}$ one has:
$$
\limsup_{n \to \infty} N_i(x,n) < +\infty,
$$
i.e., in the $\bO1$- expansion of  almost all real numbers any digit $i$ from the  alphabet $A= \mathbb{N} $ appears  only finitely many times!
\end{theorem}
\begin{proof}
Let $\Omega=[0,1], \mathcal{F}=\mathcal{B}$, and let $P=\lambda$ be the Lebesgue measure on the unit interval. For any  $i  \in \mathbb{N}$ and $k \in \mathbb{N}$ set
$$
A_k^i:=\{x: x= \bO1(g_1(x), g_2(x), \dots, g_n(x), \dots);  g_k(x)=i  \} =$$ $$= \bigcup\limits_{c_1 =1}^{\infty}...\bigcup\limits_{c_{k-1}=1}^{\infty} \bO1[c_1 c_2 ...c_{k-1}i],
$$
where $\bO1[c_1 c_2 ...c_{k-1}i]$ is the $\bO1$-cylinder. Since $\abs{\bcyl1{c_1c_2\dots c_k}} = \frac{1}{\sigma_1
\sigma_2 ... \sigma_k (\sigma_k + 1)}$,  where $\sigma_j = c_1 + c_2 +
... + c_j$, we have $\lambda(\bO1[c_1 c_2 ...c_{k-1}i]) \leq \lambda(\bO1[c_1 c_2 ...c_{k-1}1]).$

So,  $$\lambda(A_k^i)\leq \lambda(A_k^1) = \sum\limits_{c_1 =1}^{\infty}...\sum\limits_{c_{k-1}=1}^{\infty} \lambda(\bO1[c_1 c_2 ...c_{k-1}1]) = $$ $$ =\sum\limits_{c_1 =1}^{\infty}...\sum\limits_{c_{k-1}=1}^{\infty} \frac{1}{\sigma_1 \sigma_2 ... \sigma_{k-1} (\sigma_{k-1}+1)(\sigma_{k-1}+2)} = \frac{1}{2^k}. $$

Let  $A^i_{\infty} = \limsup\limits_{k \to \infty} A_k^i $. It is evident that $\sum\limits_{k=1}^{\infty} \lambda(A_k^i) \leq \sum\limits_{k=1}^{\infty} \frac{1}{2^k}=1,$ and, therefore, applying the Borel-Cantelli lemma to the sequence of  events $\{A_k^i\} ( k \in \mathbb{N})$, which are mutually depending w.r.t. the Lebesgue measure, we get $\lambda(A^i)=0$. So, for any symbol  $i \in \mathbb{N}$ and for  $\lambda$-almost all $x \in [0,1]$  the $\bO1$-expansion of $x$ contains only finitely many symbols $"i"$.
\end{proof}

  In the previous section, based on the ergodic approach, we were studying the structure of probability distributions of random variables with independent \textit{identically distributed} $\bO1$-symbols. In this Section,
   we shall study properties of the distribution of the random variable $\eta$ in general independent case, i.e., in the case where $\eta_k$ are independent but, generally speaking,  not identically distributed.

\begin{theorem}\label{theorem pro syngulyarnist' vypadkovyh ryadiv Ostrogradskogo}

Let  $\{\eta_k\}$ be a sequence of independent random variables taking values  $1,2,3,...$ with probabilities $p_{1k}, p_{2k},
p_{3k},... $ correspondingly, $~~(\sum\limits_{i=1}^{\infty} p_{ik} = 1, ~ \forall
k \in \mathbb{N})$.

If there exists a symbol "$i_0$" such that
\begin{equation}\label{umova syngul Ostrogr}
\sum_{k=1}^{\infty} p_{i_{0}k} = +\infty,
\end{equation}
then the random variable $$\eta=\sum_{k=1}^\infty
\frac{(-1)^{k-1}}{\eta_1(\eta_1+\eta_2)\dots(\eta_1+\eta_2+\dots+\eta_k)}
=\bO1(\eta_1,\eta_2,\dots,\eta_k,\dots),
$$ is singularly (w.r.t. $\lambda$) distributed.
\end{theorem}

\begin{proof} Let
$$
 A_k^{i_0}:=\{x: x= \bO1(g_1(x), g_2(x), \dots, g_n(x), \dots);  g_k(x)=i_0  \},$$ and let   $$A^i_{\infty} = \bigcup\limits_{n=1}^{\infty} \bigcap\limits_{k=n}^{\infty} A_k^i =  \limsup\limits_{k \to \infty} A_k^i.$$
The events $A_k^{i_0}, k \in \mathbb{N}$ are independent w.r.t. the probability measure $\mu_{\eta}$ and $\mu_{\eta}(A_k^{i_0})= p_{i_0 k}.$ So, by the inverse Borel-Cantelli lemma for independent events, the condition $$\sum_{k=1}^{\infty} \lambda(A_k^{i_0}) = \sum_{k=1}^{\infty} p_{i_{0}k} = +\infty$$ implies the equality $\mu_{\eta}(A^i_{\infty}) = 1$. On the other hand, from Theorem \ref{theorem  pro skinchennist povtoriv synvoliv v rozkladi Ostr-Pierce} it follows directly that $\lambda(A^i_{\infty})=0,$ which proves a mutual singularity of the measure $\mu_{\eta}$ and the Lebesgue measure.
\end{proof}

\textbf{Corollary.} If  there exists a symbol "$i_0$" such that $
\sum\limits_{k=1}^{\infty} p_{i_{0}k} = +\infty,
$ then the random variable $\eta$ with independent increments of the Ostrogradsky-Sierpi\'nski-Pierce expansion has:

1) a pure discrete distribution if and only if $\prod\limits_{k=1}^{\infty} \max\limits_i p_{ik} >0$;

2)  a singularly continuous distribution in all other cases.

\bigskip
\textbf{Remark.} Condition (\ref{umova syngul Ostrogr}) plays an important role in our proof of the singularity of $\mu_{\eta}$, but we strongly believe that the distribution of $\eta$ is orthogonal to the Lebesgue measure without any additional restrictions.

\bigskip
\textbf{Conjecture.} For any choice of the stochastic matrix $\|p_{ik}\|$ the random variable $\eta$ with independent increments of the Ostrogradsky-Sierpi\'nski-Pierce expansion is singular w.r.t. Lebesgue measure.

\bigskip
\bigskip
\textbf{ Acknowledgement } This work was partly supported by DFG 436 UKR 113/80 and
113/97 projects and by Alexander von Humboldt Foundation.


\bigskip
\begin{thebibliography}{10}
\def\selectlanguageifdefined#1{
\expandafter\ifx\csname date#1\endcsname\relax
\else\language\csname l@#1\endcsname\fi}
\ifx\undefined\url\def\url#1{{\small #1}}\else\fi
\ifx\undefined\BibUrl\def\BibUrl#1{\url{#1}}\else\fi
\ifx\undefined\BibAnnote\long\def\BibAnnote#1{}\else\fi
\ifx\undefined\BibEmph\def\BibEmph#1{\emph{#1}}\else\fi
\newcommand{\singleletter}[1]{#1}
  \def\cprime{$'$}



 \bibitem{ABPT} S. Albeverio, O. Baranovskyi, M. Pratsiovytyi, G. Torbin, The
  Ostrogradsky series and related Cantor-like sets, {\it~Acta Arithm.}, \textbf{130}(2007), no.~3. --- P.~215~--~230.


\bibitem{AKPT}  S. Albeverio, V. Koshmanenko, M. Pratsiovytyi, G. Torbin, Spectral
properties of image measures under infinite conflict interactions,
 {\it Positivity}, {\bf 8}(2004), 29-39.

\bibitem{AKT}   S. Albeverio, V. Koshmanenko, G. Torbin,  Fine
structure of the singular continuous spectrum {\it Methods Funct.
Anal. Topology.},  {\bf 9}, No. 2, (2003) 101-127 .


\bibitem{AKovPT} S. Albeverio, V.  Koval, M.  Pratsiovytyi, G. Torbin, On classification of singular measures and fractal properties of
quasi-self-affine measures  in $R^2$, {\it Random Operators and Stochastic Equations}, \textbf{16}(2008), no.~2. -- P.~181~--~211.

\bibitem{APT04} S. Albeverio, M. Pratsiovytyi, G. Torbin, Fractal probability
  distributions and transformations preserving the Hausdorff-Besicovitch
  dimension, \textit{Ergodic Th. Dynam. Sys.}, \textbf{24}(2004), no.~1, P.~1--16.

\bibitem{AT4} S.Albeverio, G.Torbin,  On fine fractal properties of generalized infinite Bernoulli convolutions,
{\it  Bulletin des Sciences Mathematiques}, \textbf{132}(2008), P. 711-727.

\bibitem{AT2} S.~Albeverio, G.~Torbin,  Fractal properties of singularly
continuous probability distributions with independent $Q^{*}$-digits,
\emph{Bulletin des Sciences Mathematiques}, \textbf{129} (2005), no.~4, 356~--~367.


\bibitem{APT3} S.~Albeverio, M.~Pratsiovytyi, G.~Torbin, Topological and fractal properties of real numbers which are not normal,
\emph{Bulletin des Sciences Mathematiques}, \textbf{129} (2005), no.~8, 615~--~630.


\bibitem{APT UMZh 2005} S.Albeverio, M.Pratsiovytyi, G.Torbin, Singular probability distributions and fractal properties of sets of real numbers defined by the asymptotic frequencies of their s-adic digits, {\it Ukrainian Mathematical Journal}, \textbf{57}(2005), 1361-1370.

\bibitem{AT1} S. Albeverio, G. Torbin,  Image measures of infinite product measures and generalized Bernoulli convolutions, {\it Transactions of the National Pedagogical University (Phys.-Math. Sci.)} {\bf 5}(2004), 248-264.



\bibitem{BPT2006}
O. Baranovskyi, M. Pratsiovytyi, G. Torbin, Matric and topological properties of sets of real numbers with restrictions on their Ostrogradsky coefficients, \textit{~Ukr. Math. J.}, \textbf{59}(2007), no. 9,  P.~1155~--~1168.


\bibitem{Bil}
P.Billingsley, \textit{Ergodic theory and information},  John Willey
and Sons, New York, 1965.


\bibitem{Bumby82}
R.~T. Bumby, Hausdorff dimensions of Cantor sets, \textit{J. Reine
Angew. Math.}, \textbf{331}(1982),- P.~192~--~206.


\bibitem{Bumby85}
R.~T. Bumby, Hausdorff dimension of sets arising in number
theory. Number Theory (New York, 1983–84) (Lecture Notes in
Mathematics, 1135). --- Springer, 1985. --- P.~1~--~8.


\bibitem{Fa}
K. Falconer, \textit{Fractal geometry --- mathematical foundations and
applications}, John Wiley, 1990.


\bibitem{Good}  J.T. Good, ~The fractional dimensional theory of continued
fractions, \textit{Proc. Cambridge Phil. Soc.}, \textbf{37}(1941), P.~199~--~228.

\bibitem{Hensley89}
 D. Hensley, The Hausdorff dimensions of some continued fraction Cantor sets, \textit{J. Number Theory}, \textbf{33}(1989), P.~182~–-~198.


\bibitem{Hensley96}
D. Hensley, A polynomial time algorithm for the Hausdorff
dimension of continued fraction Cantor sets, \textit{J. Number Theory}, \textbf{58}(1996), P.~9~–-~45.

\bibitem{Jenkinson}
O. Jenkinson, M. Policott, Computing the dimension of
dynamically defined sets: E2 and bounded continued fractions
\textit{Ergod. Th. and Dynam. Sys.}, \textbf{21}(2001), P.~1429~--~1445.

\bibitem{Khi63}
A.~Ya. Khintchine, \textit{Continued fractions}, P. Noordhoff, Ltd.,
Groningen, 1963.

\bibitem{KnKn89} A. Knopfmacher, J.Knopfmacher, Two constructions of real numbers via alternating series, \textit{Internat. J. Math. and Math.Sci.}, \textbf{12}(1989), P. 603-613.

\bibitem{Lagrange} J.L. Lagrange, Essai d' analyse num\'erique sur la transformation des fractions, \textit{Journal de l' \'Ecole Polytechnique, V\`eeme cahier,} t. II, prairial, an VI. Also in \textit{Oeuvres} ed. J.A. Serret, Tome VII, Section quatri\`eme, 291 - 313.

\bibitem{Lambert}, \textit{Verwandlung der Br\"{u}he: Beitr\"{a}ge zum Gebrauche der Mathematik und deren Anwendung}, Part 2. Berlin 1770.

\bibitem{NT_TVIMS12} R.~Nikiforov, G.~Torbin, Fractal properties of random variables with independent $Q_\infty$-digits, \emph{Theory Probab. Math. Stat.}, \textbf{86} (2013), 169~--~182.


\bibitem{PBV96}
J. Parad{\'\i}s, Ll. Bibiloni, P.  Viader, On actually
computable bijections between $\mathbb{N}$ and $\mathbb{Q}^+$,
\textit{Order} \textbf{13}(1996), no.~4, P. 369--377.

\bibitem{PVB99}
~J. Parad{\'\i}s, P. Viader, Ll.  Bibiloni, A mathematical excursion: from the three-door problem to a
Cantor-type set,  \textit{Amer. Math. Monthly} \textbf{106} (1999), no.~3, P. 241--251.




\bibitem{Pie29}
T.~A. Pierce,On an algorithm and its use in approximating
roots of algebraic equations,  \textit{Amer. Math. Monthly} \textbf{36}
(1929), no.~10, P. 523--525.




\bibitem{Rem51}
E.~Ya. Remez, On series with alternating sign which may be
connected with two algorithms of {M}.~{V}.~{O}strogradski\u\i\ for
the approximation of irrational numbers, \textit{Uspehi Matem. Nauk
(N.S.)}, \textbf{6} (1951), no.~5~(45), P. 33--42 (Russian).


\bibitem{Schweiger_book}
F. Schweiger, \textit{~Ergodic Theory of Fibred Systems and Metric Number Theory}. --- Oxford:~Clarendon Press, 1995.


\bibitem{Sha86}
J. Shallit, ~Metric theory of Pierce expansions, \textit{~Fibonacci Quart.}, \textbf{24}(1986), no.~1,-- P.~22~--~40.

\bibitem{Sie74}
\selectlanguageifdefined{french} W. Sierpi\'{n}ski,~ Sur
quelques algorithmes pour d\'{e}velopper les   nombres r\'{e}els en s\'{e}ries, \textit{Oeuvres choisies},"---
\newblock Warszawa: PWN, \textbf{1}(1974), P.~236--254.

\bibitem{Torbin 2002} G. Torbin, Fractal properties of the distributions of random variables with independent Q-symbols, {\it Transactions of the National Pedagogical University (Phys.-Math. Sci.)}, \textbf{3}(2002), 241-252.


\bibitem{TorbinUMJ2005} G. Torbin.  Multifractal analysis of singularly continuous probability
measures. \textit{Ukrainian Math. J.} \textbf{57} (2005), no. 5,
837--857.


\bibitem{Torbin SP 2007} G. Torbin,  Probability distributions with independent Q-symbols and transformations preserving the Hausdorff dimension,
{\it Theory of Stochastic Processes},\textbf{13(2007)}, 281-293.


\bibitem{VZ75}
K. Val{\=e}{\=e}v,  E. Zl{\=e}bov, The metric theory
of the Ostrogradski\u\i{} algorithm, \textit{Ukr. Math. J.} \textbf{27}
(1975), no.~1, P. 47--51.

\bibitem{VB99}  P. Viader,~ Ll.Bibiloni, J. Parad{\'\i}s, On a problem of Alfr\'ed R\'enyi, \textit{Acta Arithm.,} \textbf{91}(1999), no.2, P. 107-115.

\end{thebibliography}
\end{document}